\documentclass[11pt,a4paper]{amsart}
\usepackage[utf8]{inputenc}
\usepackage{faktor}
\usepackage[T1]{fontenc}
\usepackage{lmodern}
\usepackage[english]{babel}
\usepackage{amsmath}
\usepackage{xfrac}
\usepackage{esint}
\usepackage{stmaryrd,mathrsfs,bm,amsthm,mathtools,yfonts,amssymb,color,braket,booktabs,graphicx,graphics,amsfonts}
\usepackage{latexsym,microtype,indentfirst,hyperref}
\usepackage{xcolor}
\usepackage{courier}
\usepackage[colorinlistoftodos]{todonotes}

\title{A $W^{2, \, p}$-estimate for nearly umbilical hypersurfaces }
\author{Stefano Gioffrè}

\theoremstyle{reference}
\newtheorem{teo}{Theorem}[section]
\newtheorem{prop}[teo]{Proposition}
\newtheorem{cor}[teo]{Corollary}
\newtheorem{lemma}[teo]{Lemma}
\numberwithin{equation}{section}

\newcommand{\erre}{\mathbb{R}}
\newcommand{\esse}{\mathbb{S}}

\newcommand{\B}{\mathbb{B}}

\newcommand{\D}{\mathbb{D}}

\newcommand{\daA}[2]{\colon #1 \longrightarrow #2}

\newcommand{\restr}[2]{\left. #1 \right|_{#2}}
\newcommand{\dde}[1]{\frac{\partial}{\partial #1}}
\newcommand{\ddt}{\frac{d}{dt}}

\newcommand{\approxle}{\lesssim}

\newcommand{\Adot}{\mathring{A}}

\newcommand{\Gammag}{_{g}\Gamma}
\newcommand{\nablag}{_{g}\nabla}
\newcommand{\gu}{\mathfrak{u}}
\newcommand{\gf}{\mathfrak{f}}
\newcommand{\gh}{\mathfrak{h}}
\newcommand{\gv}{\mathfrak{v}}
\newcommand{\gw}{\mathfrak{w}}

\renewcommand{\epsilon}{\varepsilon}
\renewcommand{\phi}{\varphi}
\renewcommand{\theta}{\vartheta}
\renewcommand{\ni}{\nu}

\DeclarePairedDelimiter{\abs}{|}{|}
\DeclarePairedDelimiter{\coup}{(}{)}				
\DeclarePairedDelimiter{\cquad}{[ }{]}
\DeclarePairedDelimiter{\norm}{\lVert}{\rVert}

\DeclareMathOperator{\id}{id}

\DeclareMathOperator{\Ric}{Ric}
\DeclareMathOperator{\tr}{tr}

\DeclareMathOperator{\divv}{div}
\DeclareMathOperator{\Vol}{Vol}
\DeclareMathOperator{\hd}{HD}

\DeclareMathOperator{\grad}{grad}
\DeclareMathOperator{\osc}{osc}

\date{}
\makeindex
\allowdisplaybreaks

\begin{document}
\begin{abstract}
Let $n \ge 2$, $p$ in $ (1, \, +\infty)$ be given and let $\Sigma$ be a $n$-dimensional, closed hypersurface in $\erre^{n+1}$. Denote by $A$ its second fundamental form, and by $\Adot$ the tensor $A - \frac{\tr_g A}{n} \, g$ where $g := \restr{\delta}{\Sigma}$ and $\delta$ is the flat metric in $\erre^{n+1}$. Assuming that $\Sigma$ is  the boundary of a convex, open set we generalize the results of \cite{DLRigidity} and we prove that if the $L^p$ norm of $\Adot$ is small, then  $\Sigma$ must be $W^{2, \, p}$-close to a round sphere, with a quantitative estimate. 
\end{abstract}
\maketitle
\section{Introduction}
Let $\Sigma$ be a $n$-dimensional hypersurface in $\erre^{n+1}$. We say that a point in $q \in \Sigma$ is umbilical if its second fundamental form $A$ is diagonal when evaluated at $q$. A classical theorem in differential geometry assures that if $\Sigma$ is connected and every point $p$ is umbilical, then $\Sigma$ is a (portion of a) sphere or a (portion of a) plane. In particular if the hypersurface $\Sigma$ is closed then it must be a round sphere and it  must satisfy $A = \lambda \, g$, where $\lambda$ is a real number and $g = \restr{\delta}{\Sigma}$ is the induced metric. There have been many attempts to give a quantitative version of the rigidity theorems, especially for hypersurfaces which are boundaries of convex sets . For example in \cite{Pogorelov} it is proven that if a $2$-dimensional surface in $\erre^3$ is the boundary of a convex set and satisfies certain conditions on the ratio of the eigenvalues of the second fundamental form, then it must be close to a round sphere. More recently in \cite{DLRigidity} the authors have proven the existence of a universal  constant $C$ such that for every closed surface  $\Sigma$ in $\erre^3$ the following estimate holds:
\[
\min_{\lambda \in \erre} \norm{A - \lambda g}_{L^2_g(\Sigma)} \le C \norm{\Adot}_{L^2_g(\Sigma)}
\]
where $\Adot$ is the traceless second fundamental form. In the same work they also proved that if $\norm{\Adot}_p$ is smaller than a universal constant then there exist a conformal parametrization $\psi \daA{\esse^n}{\Sigma}$ and a vector $c= c(\Sigma)$ such that
\[
\norm{\psi  - \id - c}_{W^{2, \, 2}_\sigma(\esse^2)} \le C \norm{\Adot}_{L^2_g(\Sigma)}
\]
where $C$ is again a universal  constant. In this article we prove a stronger version of this estimate.  Our theorem works for every dimension $n \ge 2$ and for every $p \in (1, \, +\infty)$, although within the assumption that our surface is the boundary of a convex set. The proof of the theorem in \cite{DLRigidity} however is limited in dimension $2$, so in order to prove such result in every dimension, we need to find other ways. We state our result. Here:
\begin{center}
\begin{tabular}{ll}
$\Vol_n$ & $n$-dimensional Hausdorff measure.\\
$\esse^n$ & standard sphere in $\erre^{n+1}$. \\
$\sigma$ & standard metric on the sphere. \\
$g$ & restriction of the $\erre^{n+1}$-flat metric to $\Sigma$.   \\
$A$ & second fundamental form for $\Sigma$.  \\
$\Adot$ & traceless second fundamental for $\Sigma$. \\
$x \approxle_{\alpha} y$ & $x \le C y$ where $C$ is a positive constant depending only on  $\alpha$.   
\end{tabular}
\end{center}
We will also say that a hypersurface $\Sigma$ is $\delta$-\textit{admissible} if it satisfies the following conditions:
\begin{align}
\Sigma &= \partial U, \, \mbox{ where } U \mbox{ is an open, convex set} \\ 
\Vol_n(\esse^n) &= \Vol_n(\esse^n) \\
\norm{\Adot}_{L^p_g(\Sigma)} &\le \delta
\end{align}
\begin{teo}\label{Main}
Let $n \ge 2$ and $p \in (1, \, +\infty)$ be given, and let $\Sigma$ be a smooth, closed $n$-dimensional hypersurface. There exists $\delta=\delta(n, \, p)>0$ with the following property. If $\Sigma$ is $\delta$-admissible there exist a smooth parametrization $\psi \daA{\esse^n}{\Sigma}$ and a vector $c=c(\Sigma) \in U$ such that for the following estimate holds:
\begin{equation}\label{MainFormula}
\norm{\psi - \id - c}_{W^{2, \, p}_\sigma(\esse^n)} \approxle_{n, \, p} \norm{\Adot}_{L^p_g(\Sigma)}
\end{equation}
\end{teo}
From theorem \ref{Main} we infer the following corollary, which improves a result proved in \cite{DLRigidity}.
\begin{cor}
Under the assumptions of theorem \ref{Main} the following estimate holds:
\begin{equation}
\norm{g - \sigma}_{W^{1, \, p}_\sigma(\esse^n)} \approxle_{n, \, p} \norm{\Adot}_{L^p_g(\Sigma)}
\end{equation}
\end{cor}
The proof of the theorem is essentially divided into three main parts.
Firstly we show that under certain assumptions we can find a constant $\lambda$ so that is it possible control the $L^p$-norm of $A - \lambda g$ with the $L^p$-norm of $\Adot$. This estimate is basically the nonlinear version of our result.
Secondly we prove that a convex surface whose the $L^p$-norm of $\Adot$ is small is $W^{1, \, \infty}$-near to a sphere with qualitative estimates.
Thirdly we make quantitative the results obtained in the previous part, proving a certain estimate which is very close to \eqref{MainFormula}. As we will see, this estimate depends on the position of $\Sigma$ in $\erre^{n+1}$. From these three results we prove our theorem by centering $\Sigma$ properly. 
\subsection*{Notation}
Throughout this paper we will use the previous notational conventions plus the following ones:
\begin{center}
\begin{tabular}{ll}
$n$ & Integer $ \ge 2$. \\
$p$ & real number in $(1, \, +\infty)$. \\
$\B^\sigma(x)$ & geodesic ball in  $\esse^n$ centred in $x$, of radius $r$. \\
$\delta$  & standard metric in $\erre^n$. \\
$\Sigma$ & closed, $n$-dimensional hypersurface in $\erre^{n+1}$. \\
$D$ & usual derivative in $\erre^n$. \\
$\nabla$ & Levi-Civita connection associated to $\sigma$. \\
$\grad_\sigma$ & Gradient of a function defined on $\esse^n$ taken w.r.t. $\sigma$. \\
$\Delta$ & Laplace-Beltrami operator acting on the sphere. \\
$\osc(f)$ &  oscillation of $f$, $\osc(f)= \sup f - \inf f$  \\
$\Gamma(E)$ & space of smooth sections of a vector bundle $E \rightarrow M$ \\
$\faktor{\mu}{\nu}$ & density of the measure $\mu$ w.r.t. the measure $\nu$
\end{tabular}
\end{center}

\section{Proof of the main theorem}
Before entering in the details, we exhibit the parametrization on which we will work. Let us assume for a moment that $\Sigma$ is the border of an open, convex set $U$ containing $0$. We can give the following radial parametrization for $\Sigma$:
\begin{equation}\label{PsiDef}
\psi \daA{\esse^n}{\Sigma },\ \psi(x) := \rho(x) \, x := e^{f(x)} \, x
\end{equation}
Clearly $\psi$ is a smooth diffeomorphism. If $U$ does not contain $0$ we can still give such parametrization by properly translating $U$. We will say that $\Sigma$ is \textit{radially parametrized} if it can be written as the image of such $\psi$.

Now we can state the three main steps outlined in the introduction and show how these propositions easily lead to the theorem.
\begin{prop}\label{EstimateFirstOrder}
Let $\Sigma$ be a radially parametrized manifold in $\erre^{n+1}$. Then the following estimate holds:
\begin{equation}\label{EstimateFirstOrderEq}
\min_{\lambda \in \erre} \norm{A - \lambda \, g}_{L^p_g(\Sigma)} \le C \,   \norm{\Adot}_{L^p_g(\Sigma)}
\end{equation}
with $C=C(n, \, p, \, \osc(f), \, \norm{\nabla f}_\infty)$.
\end{prop}

\begin{prop}\label{ApproxTwo}
For every $0 < \epsilon < \frac{1}{4}$ there exists $0<\delta=\delta(n, \, p, \epsilon)$ with the following property. If $\Sigma$ is a $\delta$-admissible, there exist $c=c(\Sigma) \in \erre^{n+1}$ such that $\Sigma - c$ is radially parametrized and the mapping $\psi$ defined by \eqref{PsiDef} satisfies the inequalities:
 \begin{align}
 \norm{f}_\infty & \le \epsilon \label{ImportantApproxUno} \\
 \norm{\nabla f}_{\infty} &\le 2 \sqrt{\epsilon} \label{ImportantApproxDue} 
 \end{align}
\end{prop}

\begin{prop}\label{MainEstimate}
Let $\Sigma$ be a $\delta$-admissible, radially parametrized hypersurface and let its parametrization $\psi$ satisfy inequality \eqref{ImportantApproxUno}, \eqref{ImportantApproxDue}. Then the following estimate holds:
\begin{equation}\label{FinalApprox}
\norm{f - (v_f, \, \cdot) }_{W^{2, \, p}_\sigma(\esse^n)} \approxle_{n, \, p} \norm{\Adot}_{L^p_g(\Sigma)} + \sqrt{\epsilon} \norm{f}_{W^{2, \, p}_\sigma(\esse^n)}  
\end{equation}
where we have set 
\[
v_f :=   \fint_{\esse^n}   (n+1)  z f(z) \, dV_\sigma(z)
\]
\end{prop}

We show how propositions \ref{EstimateFirstOrder}, \ref{ApproxTwo} and \ref{MainEstimate} prove theorem \ref{Main}. 
\begin{proof}[Proof of theorem \ref{Main}]
Let $0<\epsilon< \frac{1}{4} $ be fixed for the moment. At the end of the argument we will choose it small enough. Let $\delta$ and $\Sigma$ be given such that $\psi$ satisfies inequalities \ref{ImportantApproxUno}, \ref{ImportantApproxDue} and \ref{FinalApprox}.
We notice that for every $c \in U$ we can define 
\begin{equation}
\psi_c \daA{\esse^n}{\Sigma - c },\ \psi_c(x) := \rho_c(x) \, x := e^{f_c(x)} \, x
\end{equation}
For every $c$ the mapping $\psi_c$ is an alternative radial parametrization for $\Sigma$, and it is a well defined diffeomorphism. We can also define:
\begin{equation}
\Phi \daA{U}{\erre^{n+1}},\ \Phi(c):= \fint_{\esse^n} (n+1)  z f_c(z) \, dV_\sigma(z)
\end{equation}
Our idea is to find $c_0 \in U$ such that $\Phi(c_0)=0$. Then we are done, because for such $f_{c_0}$ we obtain the estimate 
\[
\norm{f_{c_0}}_{W^{2, \, p}_\sigma(\esse^n)} \approxle_{n, \, p} \norm{\Adot}_{L^p_g(\Sigma)} + \sqrt{\epsilon} \, \norm{f_{c_0}}_{W^{2, \, p}_\sigma(\esse^n)}  
\]
Therefore we can find $\epsilon_0=\epsilon_0(n, \, p)$ so that the last term can be absorbed:
\[
\norm{f_{c_0}}_{W^{2, \, p}_\sigma(\esse^n)} \approxle_{n, \, p} \norm{\Adot}_{L^p_g(\Sigma)} 
\]
This estimate proves theorem \ref{Main} with $c=c_0$.

Let us find such $c_0$. By the hypothesis made we have that $\psi:=\psi_0$ satisfies the estimates \eqref{ImportantApproxUno}, \eqref{ImportantApproxDue} and \eqref{FinalApprox}. We can easily find a radius $r=r(\Sigma)>0$ such that for every $c \in \D_r$ we have that $f_c$ still satisfies such estimates. Now we work with $H$ inside the disk $\D_r$. 
We start with the following simple consideration: for every $z \in \esse^n$ there exists $x_c=x_c(z)$ in $\esse^n$ so that
\[
\psi_c(z) = \psi(x_c) - c
\]
We expand this equality and find
\[
\rho_c(z) \, z = \rho(x_c) \, x_c - c 
\]
We take the absolute value and obtain that $\rho_c$ satisfies the equality:
\[
\rho_c(z) = \abs{\rho(x_c) \, x_c - c}
\]
while $x_c=x_c(z)$ satisfies the relation
\begin{equation}\label{xc}
x_c(z) = \frac{\rho_c(z) \, z + c}{\rho(x_c(z))}
\end{equation}
Using the $W^{1, \, \infty}$-smallness, we approximate $\rho$ and $\rho_c$, and find
\begin{equation}\label{xcApp}
x_c(z) = \coup{ z + c} \coup{1 + O(\epsilon)}
\end{equation}
We approximate $f_c$:
\begin{align*}
f_c(z) &= \frac{1}{2} \, \log \abs*{ e^{f(x_c) } x_c - c }^2 = \frac{1}{2} \log \coup*{ e^{2 f(x_c)} - 2 e^{f(x_c)} (x_c, \, c) + \abs{c}^2 }^2 \\
&= \frac{1}{2} \log \coup*{1 + \abs{c}^2 - 2(x_c(z), \, c) + O(\epsilon)} \\
&= \abs{c}^2 - 2(x_c(z), \, c) + o(\abs{c}) + O(\epsilon) =  - 2 (z, \, c) +  o(\abs{c}) + O(\epsilon)
\end{align*}
This allows us to write $\Phi$ as follows:
\begin{align*}
\Phi(c) 
&= - (n+1 )\fint_{\esse^n}  (x_c(z), \, c)  \, z \, dV_\sigma + o(\abs{c}) + O(\epsilon)   \\
&= - (n+1 )\fint_{\esse^n}  ((z+c)(1 + O(\epsilon)), \, c)  \, z \, dV_\sigma + o(\abs{c}) + O(\epsilon)    \\
&=  - (n+1) \, c  + o(\abs{c}) + O(\epsilon) 
\end{align*}
Now we define $\phi:= - \frac{1}{n+1} \, \Phi$: we want to show that $0$ is in the range of $\phi$. We restrict $\phi$ to $\D_r$ and finally we choose $0<\epsilon<  \frac{1}{4}$ and $0<r$ so small that 
\[
\abs{\phi(x) - x} < \frac{1}{2} \mbox{ for every } x \in \D_r
\]
Let us argue by contradiction: suppose that $0 \notin R(\phi)$, then we can consider $\overline{\phi}:= \frac{\phi}{\abs{\cdot}} \daA{\esse^n}{\esse^n}$, and notice that 
\begin{equation}\label{Homotopic}
\abs{\overline{\phi}(x) - x} < 2 \mbox{ for every } x \in \esse^n
\end{equation}
It is easy to see that if $c_n \to c_0$ then $f_{c_n} \to f_{c_0}$ pointwise and the family $\set{f_c}_{c \in \D_r}$ is equibounded. This proves that $\Phi$ and therefore also $\phi$ are continuous. However, estimate \eqref{Homotopic} tells us that $\overline{\phi}$ is homotopic to the identity; but at the same time, we obtain that $\overline{\phi}$ is the restriction of a continuous map defined on the ball, hence it cannot be homotopic to the identity.
\end{proof}
The rest of the article will be devoted to proving the three propositions.
\section{Proofs of the propositions}
Before starting the proofs we need to report a computational lemma which shows the expression for the main geometric quantities of $\Sigma$ in the radial parametrization $\psi$. 
\begin{lemma}\label{Computations}
Let $\psi$ be as in \eqref{PsiDef}. Then we have the following expressions:
\begin{align}
g_{ij} &= e^{2 f} \coup*{\sigma_{ij} + \nabla_i f \, \nabla_j f } \label{g} \\
g^{ij} &= e^{- 2 f} \coup*{\sigma_{ij} - \frac{\nabla_i f \, \nabla_j f}{1 + \abs{\nabla f}^2} } \label{ginv} \\
\ni(x) &= \frac{1}{\sqrt{1 + \abs{\nabla f}^2}} \coup*{x - \grad_\sigma f(x)} \label{norm} \\
A_{ij} &= \frac{e^f}{\sqrt{1 + \abs{\nabla f}^2}} \coup*{ \sigma_{ij} + \nabla_i f \, \nabla_j f - \nabla^2_{ij} f} \label{Aff} \\
A^i_j 
&= \frac{e^{-f}}{\sqrt{1 + \abs{\nabla f}^2}} \coup*{ \delta^i_j - \nabla^i \nabla_j f + \frac{1}{1 + \abs{\nabla f}^2} \nabla^i f \, \nabla^2 f [\nabla f]_j} \label{Aalta} \\
dV_g &= e^{nf} \, \sqrt{1 + \abs{\nabla f}^2} \, dV_\sigma \label{Volume} \\
\Gammag_{ij}^k 
&= \Gamma_{ij}^k + \frac{1}{1 + \abs{\nabla f}^2} \nabla^2_{ij} f \, \nabla^k f + \coup*{\nabla_i f \, \delta^k_i + \nabla_j f \, \delta^k_i - \nablag^k f \, g_{ij}}   \label{Christoffel} 
\end{align}
\end{lemma}
The proof of the lemma is in the last section of the article.
\subsection{Proof of proposition \ref{EstimateFirstOrder}}
Here we prove proposition \ref{EstimateFirstOrder}. We show that the proposition follows by two lemmas. 
\begin{lemma}\label{NewLemma}
Let $\Sigma$ be a radially parametrized hypersurface and let us call $H := \frac{1}{n} g^{ij} A_{ij}$. We have the equality
\begin{equation}\label{New}
\nabla H = \frac{1}{n-1} \divv_\sigma \Adot  + \frac{n}{n-1} \, \Adot[\nabla f]
\end{equation}
\end{lemma}
\begin{lemma}\label{AnArmLemma}
Let $p \in (1, \, \infty)$, $\gu \in C^\infty(\esse^n)$, $\gf \in \Gamma(T^*\esse^n \otimes T\esse^n)$, $\gh \in \Gamma(T\esse^n)$ be given so that the following equation holds:
\[
\nabla \gu = \divv_\sigma \gf + \gf[\gh]
\]
There exists $\lambda_0 \in \erre$ such that the following estimate holds:
\begin{equation}\label{AnArm}
\norm{\gu - \lambda_0}_{L^p_\sigma(\esse^n)} \approxle_{n, \, p} \coup*{1 + \norm{\gh}_\infty} \norm{\gf}_{L^p_\sigma(\esse^n)}
\end{equation}
\end{lemma}

We show now how lemmas \ref{NewLemma} and \ref{AnArmLemma} prove proposition \ref{EstimateFirstOrder}. 
\begin{proof}[Proof of proposition \ref{EstimateFirstOrder}]
We apply lemma \ref{AnArmLemma} with $\gu=H$, $\gh= n \nabla f$ and $\gf^i_j = \frac{1}{n-1} g^{ik} \Adot_{jk}$. There exists $\lambda_0 \in \erre$ such that
\[
\norm{H - \lambda_0}_{L^p_\sigma(\esse^n)} \approxle_{n, \, p} \coup*{1 + \norm{\nabla f}_\infty} \norm{\Adot}_{L^p_\sigma(\esse^n)}
\]
We use this result and make the straight estimate:
\begin{align*}
\norm{A - \lambda_0 \, g}_{L^p_g} 
&\le \norm{\Adot}_{L^p_g} + n \norm{H - \lambda_0}_{L^p_g} \le \norm{\Adot}_{L^p_g} + n \max_{\esse^n} \abs*{\frac{dV_g}{dV_\sigma}}^{\frac{1}{p}} \norm{H - \lambda_0}_{L^p_\sigma} \\
& \approxle_{n, \, p} \norm{\Adot}_{L^p_g} + \max_{\esse^n} \abs*{\frac{dV_g}{dV_\sigma}}^{\frac{1}{p}} (1 + \norm{\nabla f}_\infty) \norm{\Adot}_{L^p_\sigma} \\
& \le (1 + \norm{\nabla f}_\infty) \max_{\esse^n} \abs*{\frac{dV_g}{dV_\sigma}}^{\frac{1}{p}}  \max_{\esse^n} \abs*{\frac{dV_\sigma}{dV_g}}^{\frac{1}{p}}  \norm{\Adot}_{L^p_g} \\
&\le (1 + \norm{\nabla f}_\infty)^{\frac{p+1}{p}} \exp \coup*{\frac{n}{p} \osc(f)} \norm{\Adot}_{L^p_g} \\
&= C \norm{\Adot}_{L^p_g(\esse^n)}
\end{align*}
where the dependence of $C$ is the following:
\begin{equation}\label{ConstantArm}
C(n, \, p, \, \osc(f), \, \nabla f)=C_{n, \, p} (1 + \norm{\nabla f}_\infty)^{\frac{p+1}{p}} \exp \coup*{\frac{n}{p} \osc(f)}
\end{equation}
This estimate completes the proof.
\end{proof}
Now we prove the two lemmas and complete this section.
\begin{proof}[Proof of lemma \ref{NewLemma}]
We firstly recall the Codazzi equation for the second fundamental form (see \cite{GHL} for a proof):
\begin{equation}\label{Codazzi}
\nablag_k A^i_j = \nablag_j A^i_k
\end{equation}
Equation \eqref{Codazzi} however holds for the Levi-Civita connection $\nablag$ taken with respect to the metric $g$, while we need to find a formula for the $\sigma$ connection $\nabla$.
So we firstly expand $\nablag A$:
\begin{align*}
\nablag_k A^i_j
&= D_k A^i_j + \Gammag_{kl}^i A^l_j - \Gammag_{kj}^l A^i_l \\
\end{align*}
Now we plug this expression into \eqref{Codazzi}, and use the expression \eqref{Christoffel} for the Christoffel symbols obtaining
\begin{align*}
\nabla_k A^i_j 
&= \nabla_j A^i_k + \frac{\nabla^i f}{1 + \abs{\nabla f}^2} \coup*{ \nabla^2_{jl} f A^l_k - \nabla^2_{jk} f A^i_l } + \\ 
&  +  \coup*{\nabla_j f \, \delta^i_l + \nabla_l \, f \delta^i_j - \nablag^i f \,  g_{jl}} A^l_k -  \coup*{\nabla_k f \, \delta^i_l + \nabla_l f \, \delta^i_k - \nablag^i f \, g_{kl}}A^l_j
\end{align*}
We now notice  that $ \nabla^2_{jl} f A^l_k = \nabla^2_{kl} f A^l_j$. Expanding the term in fact we have
\begin{align*}
\nabla^2_{jl} f A^l_k 
&= \frac{e^{-f} \,  \nabla^2_{jl} f}{1 + \abs{\nabla f}^2} \coup*{ \delta^l_k - \nabla^l \nabla_k f + \frac{1}{1 + \abs{\nabla f}^2} \nabla^2 f[\nabla f]_j \, \nabla^l f } \\
&= \frac{e^{-f}}{1 + \abs{\nabla f}^2} \coup*{ \nabla^2_{jk} f - \coup*{\nabla^2 f \, \nabla^2 f}_{jk}  + \frac{  \nabla^2 f [\nabla f]_j  }{1 + \abs{\nabla f}^2} \, \nabla^2 f[\nabla f]_k   } \\ 
&= \frac{e^{-f}}{1 + \abs{\nabla f}^2} \coup*{ \nabla^2_{jk} f - \coup*{\nabla^2 f \, \nabla^2 f}_{kj}  + \frac{ \nabla^2 f [\nabla f]_k  }{1 + \abs{\nabla f}^2}  \, \nabla^2 f[\nabla f]_j   } \\
&=\nabla^2_{kl} f A^l_j
\end{align*}
This allows us to simplify the equation, obtaining 
\begin{align*}
\nabla_k A^i_j 
&= \nabla_j A^i_k + \coup*{\nabla_j f \, \delta^i_l + \nabla_l \, f \delta^i_j } A^l_k - \coup*{\nabla_k f \, \delta^i_l + \nabla_l f \, \delta^i_k }A^l_j 
\end{align*}
We track the indexes $i$ and $j$:
\begin{align*}
\nabla_k A^i_i 
&= \nabla_i A^i_k + \coup*{\nabla_i f \, \delta^i_l + \nabla_l \, f \delta^i_i } A^l_k - \coup*{\nabla_k f \, \delta^i_l + \nabla_l f \, \delta^i_k }A^l_i \\
&= \nabla_i A^i_k + n A^l_k \, \nabla_l f - A^i_i \, \nabla_k f
\end{align*}
Finally we complete the proof. Indeed we write 
\[
A^i_j = \Adot^i_j + \frac{1}{n} A^l_l \, \delta^i_j = \Adot^i_j + H \, \delta^i_j
\] 
With this expression we obtain
\begin{align*}
(n-1) \nabla_k H
&= \nabla_i \Adot^i_k + n \Adot^l_k \, \nabla_l f 
\end{align*}
The thesis follows dividing by $n-1$.
\end{proof}

\begin{proof}[Proof of lemma \ref{AnArmLemma}]
Using normal coordinates and the symmetries of the sphere, it is easy to show that for every $0<\epsilon$ there exist $0<R=R(\epsilon)$ such that for every $x \in \esse^n$ the following estimates hold in $B^\sigma_R(x)$:
\begin{equation}\label{SpherEstimates}
\begin{cases}
(1 - \epsilon) dV_\sigma \le dx \le (1+\epsilon)dV_\sigma \\
\abs{g_{ij} - \delta_{ij}} \le \epsilon^2 \\
\abs{\Gamma^k_{ij}} \le \epsilon \\
\abs{D_k g_{ij}} \le \epsilon
\end{cases}
\end{equation}
Using these coordinates we localize the expression for $\gu$. In fact we write in local chart:
\begin{align*}
\divv_\sigma \gf_k 
&= g^{ij} \nabla_i \gf_{jk} = g^{ij} \coup*{ D_i \gf_{jk} + \Gamma^l_{ij} \gf_{lk} + \Gamma^l_{ik} \gf_{jl} } 
\\
&= D_i \coup*{ g^{ij} \gf_{jk} } + g^{ij }\coup*{  \Gamma^l_{ij} \gf_{lk} + \Gamma^l_{ik} \gf_{jl} }   - D_i g^{ij} \gf_{jk} \\
&=: \divv_\delta \tilde{\gf} + O_\epsilon(\gf)
\end{align*}
where we have denoted by $\divv_\delta$ the flat divergence $\divv X_j = D^i X_{ij}$, $\tilde{\gf}$ and by $O_\epsilon(\gf)$ a quantity which satisfies the estimate
\[
\abs{O(\gf)} \le C(\esse^n) \epsilon \abs{\gf}
\]
In summary we have obtained:
\[
D\gu = \divv_\delta \tilde{\gf} + \gf[\gh] + O_\epsilon(\gf) \mbox{ in } B^\sigma_R(x)
\]
We notice that $\tilde{\gf}$ satisfies the estimate $\abs{\tilde{\gf} - \gf} \le \epsilon$. Now we write $\gu = \gv + \gw$, with $\gv$ and $\gw$ satisfying the conditions:
\[
\begin{cases}
\Delta_\delta \gv = \divv_\delta \divv_\delta \tilde{\gf} \\
\restr{\gv}{\partial B_R^\sigma(x)} = \restr{\gu}{\partial B_R^\sigma(x)}
\end{cases}
\]
and 
\[
\begin{cases}
\Delta_\delta \gw = \divv \coup*{ \, \gf[h] + O_\epsilon(\gf) \, } \\
\restr{\gw}{\partial B_R^\sigma(x)} = 0
\end{cases}
\]
where $\Delta_\delta$ is the flat laplacian. 
The first system is studied in \cite{Daniel} where the author proves the existence of a real number $\lambda $ such that the following estimate holds: 
\[
\norm{\gv - \lambda(x)}_{L^p_\delta(B_{\faktor{R}{4}}(x))} \approxle_{n, \, p} \norm{\gf}_{L^p_\delta(B_R(x))}
\]
The second system is well known. In \cite{Ambrosio} it is shown the inequality:
\[
\norm{\gw}_{L^p_\delta(B_R(x))} \approxle_{n, \, p} \norm{\gf[h] + O_\epsilon(\gf)}_{L^p_\delta(B_R(x))} \approxle \coup*{\norm{h}_\infty + \epsilon} \, \norm{\gf}_{L^p_\delta(B_R(x))} 
\]
We patch together the two estimates. Choosing $\epsilon$ sufficiently small, say $\epsilon = \frac{1}{2}$, we find a radius $R$ and a constant $0<C= C(n, \, p)$ such that for every $x \in \esse^n$ there exists $\lambda(x) \in \erre$ which satisfies the estimate:
\begin{equation}\label{Local}
\norm*{\gu - \lambda(x) }_{L^p_\sigma\coup{B^\sigma_{\faktor{R}{4}}(x)}} \le C (1 + \norm{\nabla f}_\infty) \norm{\Adot}_{L^p_\sigma(B^\sigma_R(x))}
\end{equation}
Now we have to make this estimate global. We follow a technique used in \cite{Daniel} and prove the following lemma.
\begin{lemma}
Suppose $\gu \in C^\infty(\esse^n)$ has the following property. There is a radius $\rho$ such that for every $x \in \esse^n$ the local estimate is satisfied:
\begin{equation}\label{LocalDue}
\norm{\gu - \lambda(x) }_{L^p_\sigma(B_{r}(x))} \le  \beta
\end{equation}
where $\lambda(x)$ is a real number depending on $x$, $r \le 2\rho$ and $\beta$ does not depend on $x$. Then $\gu$ satisfies the global estimate:
\[
\norm{\gu - \lambda }_{L^p_\sigma(\esse^n)} \le C \beta
\]
where $\lambda \in \erre$ and $C = C(n, \, p, \, \rho)$ is a positive constant. 
\end{lemma}
\begin{proof}
We choose a finite covering of balls $\set{( B_j, \, \lambda_j)}_{j=1}^N$ which satisfies the following properties. Every ball $B_j$ has radius $2\rho$, estimate \eqref{LocalDue} holds with $\lambda_j$, and for every $j$, $k$ there exists a ball of radius $\rho$ contained in $B_j \cap B_k$.  Therefore, given two balls $B_j$ and $B_k$ whose intersection is non empty, we have: 
\begin{align*}
\abs{\lambda_j - \lambda_k} 
&= \frac{1}{\Vol_n(B_j \cap B_k)^{\frac{1}{p}}} \norm{\lambda_j - \lambda_k}_{L^p_\sigma(B_j \cap B_k)} \\ 
&= \frac{1}{\Vol_n(B_j \cap B_k)^{\frac{1}{p}}} \norm{\lambda_j - \gu + \gu - \lambda_k}_{L^p_\sigma(B_j \cap B_k)} \\
& \le \frac{1}{\Vol_n(B_j \cap B_k)^{\frac{1}{p}}} \coup*{ \norm{ \gu - \lambda_k}_{L^p_\sigma(B_j \cap B_k)} + \norm{\gu - \lambda_k}_{L^p_\sigma(B_j \cap B_k)} } \\
&\le \frac{2 \, \beta}{\Vol_n(B_j \cap B_k)^{\frac{1}{p}}}
\end{align*}
Using the properties of the covering we obtain
\[
\abs{\lambda_j - \lambda_k} \le 2 \Vol_n(B_\rho)^{- \frac{1}{p}} \beta
\]
The volume of the ball $B_\rho$ does not depend on the center because of the symmetry of the sphere. Define $\lambda_{\min} := \min_{1 \le j \le n} \lambda_j$ and  $ \lambda_{\max} := \max_{1 \le j \le n} \lambda_j$. Consider a path joining the ball in the cover with $\lambda_{\min}$ to the one with $\lambda_{\max}$. Since this path can cross at most $N$ different balls, we obtain 
\[
\abs{\lambda_{\max} - \lambda_{\min}} \le 2 N  \Vol_n(B_\rho)^{- \frac{1}{p}}  \beta = C(n, \, p, \, \rho) \, \beta
\]
For every $\lambda_{\min} \le \lambda \le  \lambda_{\max}$ we have
\begin{align*}
\norm{\gu - \lambda}_{L^p_\sigma(\esse^n)}
&\le \sum_{j=1}^N \norm{\gu - \lambda}_{L^p_\sigma(B_j)} \\
&\le \sum_{j=1}^N \norm{\gu - \lambda_j + \lambda_j - \lambda}_{L^p_\sigma(B_j)} \\
&\le \sum_{j=1}^N \norm{\gu - \lambda_j}_{L^p_\sigma(B_j)} + \abs{\lambda_j - \lambda} \Vol_n(B_j)^{-\frac{1}{p}} \\
&\le \sum_{j=1}^N \norm{\gu - \lambda_j}_{L^p_\sigma(B_j)} + \abs{\lambda_{\max} - \lambda_{\min}} \Vol_n(B_j)^{-\frac{1}{p}} \\
&\le C_2(n, \, p, \, \rho) \, \beta
\end{align*}
and the lemma follows. 
\end{proof}
We apply this lemma with $\beta= \norm{\Adot}_{L^p_\sigma}$ and find the thesis.
\end{proof}

\subsection{Proof of proposition \ref{ApproxTwo}}
We prove here proposition \ref{ApproxTwo}. We start with the $L^{\infty}$-bound.
\subsubsection{Bound for $f$}
Before starting the proof we recall the a proposition proved in \cite{Daniel} which shows how the smallness of $\norm{\Adot}_p$ implies the nearness of $\Sigma$ to a sphere.\footnote{Actually, the statement in \cite{Daniel} is different. However, as the author points out in (\cite{Daniel}, p.$22$), it is easy to see that assuming $\Sigma$ to be the boundary of a convex set leads to this formulation. }:
\begin{prop}\label{Near}
Let $n \ge 2$, $p \in (1, \, +\infty)$ be given, and let $\Sigma^n \subset \erre^{n+1}$ be a $n$-dimensional, closed hypersurface with induced Riemannian metric $g$, satisfying
 \begin{itemize}
\item$\Sigma=\partial U$ where $U$ is an open, convex set.
\item $\Vol_n(\Sigma) = \Vol_n(\esse^n)$
\end{itemize}
For every $\epsilon>0$ there exist $\delta>0$ depending on $n$, $p$, $\epsilon$, such that
\begin{equation}
\norm{\Adot}_{L^p_g(\Sigma)} \le \delta \Rightarrow  d_{\hd}(\Sigma, \, \esse^n(c)) \le \epsilon \mbox{ for some } c \in U
\end{equation}
where $d_{\hd}$ is the Hausdorff distance between two sets, that is 
\[
d_{\hd}(A, \, B):= \max \set{ \sup_{x \in A} d(x, \, B), \, \sup_{y \in B} d(y, A)  }
\]
\end{prop}
Using \ref{Near} we immediately deduce our $L^\infty$-nearness.
\begin{proof}
We assume $c=0 $ without loss of generality. By proposition \ref{Near}, for every $\epsilon>0$ there exists $\delta=\delta(n , \, p, \, \epsilon)>0$ such that
\[
\norm{\Adot}_{L^p_g(\Sigma)} \le \delta \Rightarrow d_{\hd}(\Sigma, \, \esse^n) \le \epsilon 
\]
Let us assume $\norm{\Adot}_p \le \delta$: this immediately implies
\begin{equation}\label{rho}
\abs{\rho(x) - 1} = \abs{\rho(x) \, x - x} \le d_{\hd} \coup*{\esse^n, \, \Sigma} \le \epsilon
\end{equation}
By equation \eqref{rho} we obtain
\[
\abs*{f(x) \int_0^1 \rho(x)^t \, dt} \le \epsilon
\]
Now using the estimate on $\rho$ we have
\[
\abs*{f(x) \int_0^1 (1 - \epsilon)^t \, dt} \le \epsilon
\]
and solving the integral
\[
\abs{f(x)} \le \abs{\log (1 - \epsilon)} \le \epsilon
\]
which is the thesis.
\end{proof}
\subsubsection{Bound for $\nabla f$} 
\begin{proof}
This bound actually does not depend on any other assumptions on $\Sigma$ except than the oscillation of $f$. In fact, it follows immediately by the following lemma:
\begin{lemma}
Let $\Sigma$ be a convex, radially parametrized hypersurface in $\erre^{n+1}$. The following inequality holds: 
\[
\norm{\nabla f}^2_\infty \le 2 \osc(f) \coup*{1 + \norm{\nabla f}^2_\infty}
\]
In particular if $\osc(f) < \frac{1}{2}$ we find the estimate:
\begin{equation}
\norm{\nabla f}_\infty \le \sqrt{\frac{\osc (f)}{1 - 2 \osc(f)}}
\end{equation}
\end{lemma}
\begin{proof}
In \cite{GHL} it is shown that a hypersurface is the boundary of a convex, open set iff its second fundamental form satisfies the inequality $A \ge 0$. By lemma \ref{Computations} we obtain that $\Sigma$ is convex iff $f$ satisfies the inequality 
\begin{equation}\label{Convexityff}
\nabla^2 f \le \sigma + \nabla f \otimes  \nabla f
\end{equation}
Consider a point $x_0 \in \esse^n$, and an unit vector $\xi$ in $ T_x \esse^n$ which satisfies $\coup*{ \nabla f(x_0), \, \xi } = - \norm{\nabla f}_\infty$. Setting $x_\tau = \exp_{x_0}(\tau \xi)$ the lemma follows by the simple equality
\[
f(x_\tau) - f(x_0) = \coup*{\nabla f(x_0), \, \tau \xi} + \int_0^1 t \int_0^1 \nabla^2 f (\gamma(st))[\dot{\gamma}(st), \, \dot{\gamma}(st)] \, dsdt
\]
where $\gamma \daA{[0, \, 1]}{\esse^n}$ is the geodesic which connects $x_0$ and $x_\tau$. Applying \eqref{Convexityff} we find
\begin{align*}
f(x_\tau) - f(x_0) 
&\le \coup*{\nabla f(x_0), \, \tau \xi} + \frac{\tau^2}{2} \coup*{1 + \norm{\nabla f}^2_\infty} \\
&= - \tau \norm{\nabla f}_\infty + \frac{\tau^2}{2} \coup*{1 + \norm{\nabla f}^2_\infty}
\end{align*}
Finally we obtain the inequality
\[
\norm{\nabla f}_\infty \le \frac{\osc(f)}{\tau} + \frac{\tau}{2} \coup*{1 + \norm{\nabla f}^2_\infty} \mbox{ for every } \tau>0
\]
Choosing $\tau= \sqrt{2 \, \osc(f) \, (1 + \norm{\nabla f}_\infty}$ we obtain the result.
\end{proof}
Now that we have proven this result, we notice that for $0<\epsilon < \frac{1}{4}$ we find our thesis.
\end{proof}

\subsection{Proof of proposition \ref{MainEstimate}}
We deal with proposition \ref{MainEstimate}. The idea of the proof is to see proposition \ref{EstimateFirstOrder} as a nonlinear version of an elliptic estimate and then linearise it. For simplicity, we have split the proof in two parts.
\subsubsection{Linearisation}
In this section we prove the following proposition. 
\begin{prop}
Let $0<\epsilon<\frac{1}{4}$ be given, and let $\delta$ be as in proposition \ref{ApproxTwo}. If $\Sigma$ is a $\delta$-admissible, radially parametrized sphere, then the following estimate is true:
\begin{equation}\label{Linearised}
\norm{\Delta f + n f  }_{L^p_\sigma} \approxle_{n, \, p} \norm{\Adot}_{L^p_g} + \sqrt{\epsilon}\norm{f}_{W^{2, \, p}_\sigma}
\end{equation}
\end{prop}
\begin{proof}
We need the following corollary of theorem \ref{EstimateFirstOrder} , which will be proved in the last section of this article.
\begin{cor}\label{FirstMean}
The following inequality holds:
\[
\norm{A - \overline{H} \, g}_{L^p_g(\Sigma)} \le (n+1) \,  \min_{\lambda \in \erre} \norm{A - \lambda \, g}_{L^p_g(\Sigma)}
\]
where  $\overline{H}$ is the average of $H$ 
\[
\overline{H} := \fint_{\esse^n} H \, dV_g
\]
\end{cor}
This result give us the following estimate:
\[
\norm{A - \overline{H} \, g}_{L^p_g(\Sigma)} \le C \norm{\Adot}_{L^p_g(\Sigma)}
\]
where $C$ is given by \eqref{ConstantArm}. We notice that \textit{a priori} $C$ depends on $n$, $p$, $\osc(f)$ and $\nabla f$, but thanks to proposition \ref{ApproxTwo} we can remove the dependences on $\osc(f)$ and $\nabla f$ by using the $W^{1, \, \infty}$-smallness.  Therefore we can write:
\begin{equation}\label{MeanEst}
\norm{A - \overline{H} \, g}_{L^p_g(\Sigma)} \approxle_{n, \, p}  \norm{\Adot}_{L^p_g(\Sigma)}
\end{equation}
Let $0<\epsilon<\frac{1}{4}$ and $\delta$ be chosen so by proposition \ref{ApproxTwo} inequalities \eqref{ImportantApproxUno}, \eqref{ImportantApproxDue} are true. We will show how these two inequalities lead to the conclusion. 
Before entering in the details, we need to fix a notation: we will write $O_{\epsilon^\gamma}(\norm{f}_{k, \, p})$ to denote a quantity which satisfies the estimate:
\[
\norm{O_{\epsilon^\gamma} (\norm{f}_{k, \, p})}_{k, \, p} \approxle_{n, \, p} \epsilon^{\gamma} \norm{f}_{k, \, p}
\]

Now we can start the proof. From \eqref{MeanEst} we obtain 
\[
\norm{A - g}_{L^p_g} \approxle_{n, \, p} \norm{\Adot}_{L^p_g} + \abs{\overline{H} - 1}
\]
We simplify the left hand side. We recall formula \eqref{Aff}:
\[
A_{ij} = \frac{e^f}{\sqrt{1 + \abs{\nabla f}^2}} \coup*{ \sigma_{ij} + \nabla_i f \, \nabla_j f - \nabla^2_{ij} f}
\]
This formula can be strongly simplified using the $W^{1, \, \infty}$-smallness of $f$. Indeed, we have
\[
\frac{1}{\sqrt{1 + \abs{\nabla f}^2}} -1 = \int^1_0 \ddt \frac{1}{\sqrt{1 + t^2\abs{\nabla f}^2}} \, dt = O_{\sqrt{\epsilon}}(\norm{\nabla f}_p)
\]
Therefore we obtain 
\begin{align*}
A_{ij} 
&= e^f  \coup*{ \sigma_{ij} + \nabla_i f \, \nabla_j f - \nabla^2_{ij} f} + O_{\sqrt{\epsilon}}(\norm{\nabla f}_{1, \, p}) \\
&= e^f  \coup*{ \sigma_{ij} - \nabla^2_{ij} f} + O_{\sqrt{\epsilon}}(\norm{\nabla f}_{1, \, p}) 
\end{align*}
We use the same idea for simplifying the exponential: by standard calculus, we find
\[
e^{f} = 1 + f + O_\epsilon (\norm{f}_p)
\]
These simplifications give us the approximated second fundamental form:
\begin{equation}\label{SimpleAff}
A_{ij} = \sigma_{ij} - \nabla^2_{ij} f + f \sigma_{ij} + O_{\sqrt{\epsilon}}(\norm{ f}_{2, \, p})
\end{equation}
With the same ideas we find also the approximated metric
\begin{equation}\label{SimpleG}
g_{ij} = (1 + 2 f+ O_{\sqrt{\epsilon}}(\norm{f}_{1, \, p}) )  \sigma_{ij} 
\end{equation}
and its inverse 
\begin{equation}\label{SimpleGInv}
g^{ij} = (1 - 2 f+ O_{\sqrt{\epsilon}}(\norm{f}_{1, \, p}) )   \sigma^{ij} 
\end{equation}
Via formulas \eqref{SimpleAff} and \eqref{SimpleG} we find the approximated equation:
\[
A - g = - f \sigma - \nabla^2 f + O_{\sqrt{\epsilon}}(\norm{f}_{2, \, p})
\]
We prove that the term $\abs{\overline{H} - 1}$ is negligible. Firstly we show how formulas \eqref{SimpleGInv} and \eqref{SimpleAff} give us approximated expression for the mean curvature. 
\begin{align*}
g^{ij} A_{ij} 
&= (1 - 2f + O_{\sqrt{\epsilon}} (\norm{f}_{1, \, p}) )  \sigma^{ij}  \coup{ \sigma_{ij} - \nabla^2_{ij} f + f \sigma_{ij} + O_{\sqrt{\epsilon}}(\norm{ f}_{2, \, p})} \\
&= n - nf - \Delta f +  O_{\sqrt{\epsilon}}(\norm{f}_{2, \, p})
\end{align*}
We obtain the approximated mean curvature:
\begin{equation}\label{SimpleMean}
H = 1 - f - \frac{1}{n} \Delta f +  O_{\sqrt{\epsilon}}(\norm{f}_{2, \, p})
\end{equation}
Now we approximated $\overline{H}$:
\begin{align*}
\overline{H} &= \fint_{\esse^n} H \, dV_g = \fint_{\esse^n} \coup*{ 1 - f - \frac{1}{n} \Delta f +  O_{\sqrt{\epsilon}}(\norm{f}_{2, \, p}) } e^{nf} \sqrt{1 + \abs{\nabla f}^2} \, dV_\sigma \\
&= \fint_{\esse^n} e^{nf} \coup*{1 - f} dV_\sigma +  O_{\sqrt{\epsilon}}(\norm{f}_{2, \, p}) \\
& = 1 - (n-1) \fint_{\esse^n} f dV_\sigma +  O_{\sqrt{\epsilon}}(\norm{f}_{2, \, p}) 
\end{align*}
We have found the approximated average of the mean curvature:
\begin{equation}\label{SimpleF}
\overline{H}= 1 - (n-1) \fint_{\esse^n} f \, dV_\sigma +  O_{\sqrt{\epsilon}}(\norm{f}_{2, \, p}) 
\end{equation}
We show that the average of $f$ is actually negligible. Indeed, since $\Sigma$ is $\delta$-admissible it satisfies the volume condition
\[
\Vol_n(\Sigma) = \Vol_n(\esse^n)
\]
However, by the volume formula \ref{Volume} the condition means
\[
\fint_{\esse^n} e^{nf} \sqrt{1 + \abs{\nabla f}^2} \, dV_\sigma = 1
\]
With the previous approximations, we find 
\[
0 = \fint_{\esse^n} e^{nf} \sqrt{1 + \abs{\nabla f}^2} \, dV_\sigma - 1 = n \fint_{\esse^n} f dV_\sigma +  O_{\sqrt{\epsilon}}(\norm{f}_{1, \, p}) 
\]
This means 
\[
\fint_{\esse^n} f dV_\sigma =  O_{\sqrt{\epsilon}}(\norm{f}_{1, \, p}) 
\]
Finally we infer 
\begin{equation}\label{SimpleAv}
\abs{\overline{H} - 1} = O_{\sqrt{\epsilon}}(\norm{f}_{2, \, p}) 
\end{equation}
Using formulas \eqref{SimpleAff}, \eqref{SimpleF}, \eqref{SimpleAv} we obtain the approximated version of inequality of \eqref{MeanEst}:
\begin{equation}\label{Linearis}
\norm{\nabla^2 f + f \sigma }_{L^p_\sigma(\esse^n)} \approxle_{n, \, p} \norm{\Adot}_{L^p_\sigma(\Sigma)} + \sqrt{\epsilon}\norm{f}_{W^{2, \, p}_\sigma(\esse^n)}
\end{equation}
We obtain the thesis by simply applying the Cauchy-Schwartz inequality:
\[
\abs{\Delta f + n f} = \abs{\coup*{ \nabla^2 f + f \sigma, \, \sigma }} \le n \abs{ \nabla^2 f + f \sigma}
\]
\end{proof}

 \subsubsection{Conclusion}
We prove the following proposition.
\begin{prop}\label{ObataLike}
Let $f \in C^\infty(\esse^n)$. Then, the following estimate holds:
\begin{equation}
\norm{f - (v_f, \, \cdot)}_{W^{2, \, p}(\esse^n)} \approxle_{n, \, p} \norm{\Delta f + n f}_{L^p(\esse^n)}
\end{equation}
\end{prop}
\begin{proof}
Proposition \ref{ObataLike} completes the theorem. 
Firstly we recall the main ingredient of this part, the Obata theorem (see \cite{Obata}): 
\begin{teo}\label{Obata}
Let $(M, \, g)$ be a closed manifold which satisfies the following condition on the Ricci tensor $\Ric$:
\begin{equation}
\Ric(X, \, X) \ge (n-1) \, g(X, \, X) \mbox{ for every vector field } X
\end{equation}
and the Laplacian condition:
\begin{equation}
-\Delta_g f = n  f \mbox{ for some } f
\end{equation}
Then $M$ is isometric to the round sphere $(\esse^n, \, \sigma)$ and we also have
\begin{equation}
\ker - \Delta_\sigma - n = \set{ \phi_v  \mid \phi_v(x) = (v, \, x),\ v\in \erre^{n+1}}
\end{equation}
\end{teo}
Obata's result give us the equality
\begin{align*}
\norm*{ \Delta + n f}_{L^p_\sigma(\esse^n)} = \inf_{ \, v \in \erre^{n+1}} \norm{f  - (v, \, \cdot)}_{W^{2, \, p}_\sigma(\esse^n)}
\end{align*}
What remains is to prove the following estimate:
\begin{equation}\label{LastEstimate}
\norm{f - (v_f, \,  \cdot)}_{W^{2, \, p}_\sigma(\esse^n)} \approxle_{n, \, p} 
\inf_{v \in \erre^{n+1}} \norm{f  -  (v, \,  \cdot)}_{W^{2, \, p}_\sigma(\esse^n)}
\end{equation}
In order to achieve this result, we define the alternative Sobolev norm
\begin{equation}
\abs{f}_{W^{2, \, p}_\sigma(\esse^n)}^p := \norm{f}_{L^p_\sigma(\esse^n)} + \norm{\nabla^2 f}_{L^p_\sigma(\esse^n)}
\end{equation}
The alternative Sobolev norm is equivalent to the standard Sobolev norm. Moreover, the equivalence constants depend only on the geometry of the sphere, hence on $n$ and $p$. 
Due to the equivalence of the two norms, we have 
\[
\inf_{v \in \erre^{n+1}} \abs*{f - (v, \, \cdot)}_{W^{2, \, p}_\sigma(\esse^n)} \approxle_{n, \, p}  \inf_{v \in \erre^{n+1}} \norm*{f - (v, \, \cdot)}_{W^{2, \, p}_\sigma(\esse^n)}
\]
therefore, we can work with the alternative Sobolev norm without loss of generality.
We notice that for any $v$ we have the following inequality:
\[
 \abs*{f - (v_f, \, \cdot)}_{W^{2, \, p}_\sigma(\esse^n)} 
 \le  \abs*{f - (v, \, \cdot)}_{W^{2, \, p}_\sigma(\esse^n)} +  \abs*{  ( v_f - v , \, \cdot  )}_{W^{2, \, p}_\sigma(\esse^n)} 
\]
We notice that if $v$ has unit norm, then 
\[
\abs{(v, \, \cdot)}_{2, \, p}^p = (n+1) \norm{(v, \, \cdot)}_{2, \, p} = (n+1) \norm{(e_1, \, \cdot)}_{2, \, p} = c_{n, \, p}
\]
and this gives us the estimate 
\begin{equation}\label{Lastlast}
 \abs*{f - (v_f, \, \cdot)}_{W^{2, \, p}_\sigma(\esse^n)}  \approxle_{n, \, p} \abs*{f - (v, \, \cdot)}_{W^{2, \, p}_\sigma(\esse^n)} +  \abs*{ v_f - v }  \mbox{ for any } v
\end{equation}
What remains is to study $\abs{v_f - v_m}$, where $v_m$ is the minimum of the function $v \longmapsto \abs{f - (v, \, \cdot)}_{2, \, p}$. Given any $v$, we have
\begin{align*}
\abs{v_{f, \, i} - v_i}
&= 
\abs*{ (n+1)\fint_{\esse^n} x_i f \, dV_\sigma - v_i} \\
&= \frac{1}{\Vol_n(\esse^n)} \abs*{ (n+1) \int_{\esse^n} x_i f \, dV_\sigma - \Vol_n(\esse^n) \, v_i}  \\
&= \frac{1}{\Vol_n(\esse^n)} \abs*{ (n+1) \int_{\esse^n} x_i f \, dV_\sigma - (n+1)\coup*{\int_{\esse^n} \abs{x_i}^2 \, dV_\sigma} \, v_i}   \\
&= \frac{1}{\Vol_n(\esse^n)} \abs*{ (n+1) \int_{\esse^n} x_i f \, dV_\sigma - (n+1) \int_{\esse^n} (v, \, x) x_i \, dV_\sigma } \\
&= \abs*{ (n+1) \fint_{\esse^n} x_i f \, dV_\sigma - (n+1) \fint_{\esse^n} (v, \, x) x_i \, dV_\sigma }
\end{align*}
where we have used the simple equality
\[
\int_{\esse^n} \abs{x_i}^2 dV_\sigma = \frac{\Vol_n(\esse^n)}{n+1} \mbox{ for every } i= 1 \dots n+1
\]
Recalling that for every $v \in \erre^{n+1}$ we have $\Delta (v, \, \cdot) + n (v, \, \cdot) = 0$, we find 
\begin{align*}
\abs{v_{f, \, i} - v_i}
&= \abs*{ (n+1) \fint_{\esse^n} x_i f \, dV_\sigma - (n+1) \fint_{\esse^n} (v, \, x) x_i \, dV_\sigma } \\
&= \abs*{ (n+1) \fint_{\esse^n} x_i \coup*{ f - (v, \, \cdot)} \, dV_\sigma } \\
&=  \abs*{ \fint_{\esse^n} x_i \coup*{ f - (v, \, \cdot)} \, dV_\sigma  - \fint_{\esse^n} \Delta x_i \coup*{ f - (v, \, \cdot)} \, dV_\sigma } \\
&=  \abs*{ \fint_{\esse^n} x_i \coup*{ f - (v, \, \cdot)} \, dV_\sigma  - \fint_{\esse^n} x_i \Delta \coup*{ f - (v, \, \cdot)} \, dV_\sigma } \\
&\le  \fint_{\esse^n} \abs*{{ f - (v, \, \cdot)}} \, dV_\sigma  - \fint_{\esse^n} \abs*{\Delta \coup*{ f - (v, \, \cdot)}} \, dV_\sigma \\ 
&\approxle_{n, \, p} \norm{f - (v, \, \cdot)}_p + \norm{\Delta \coup{f - (v, \, \cdot)}}_p \\
&\le  \norm{f - (v, \, \cdot)}_p + \norm{\nabla^2 \coup{f - (v, \, \cdot)}}_p = \abs*{f - (v, \, \cdot)}_{2, \, p}
\end{align*}
This estimate works for every $v$, hence we find 
\[
\abs{v_f - v_m}_{2, \, p} \approxle_{n, \, p} \, \abs{f - (v_m, \, \cdot)}_{2, \, p}
\]
We can improve \eqref{Lastlast}, obtaining 
\begin{align*}
 \abs*{f - (v_f, \, \cdot)}_{W^{2, \, p}_\sigma(\esse^n)} 
 &\le  \abs*{f - (v_m, \, \cdot)}_{W^{2, \, p}_\sigma(\esse^n)} +  \abs*{  ( v_f - v_m , \, \cdot  )}_{W^{2, \, p}_\sigma(\esse^n)}\\
 &\approxle_{n, \, p} \abs*{f - (v_m, \, \cdot)}_{W^{2, \, p}_\sigma(\esse^n)} +  \abs*{ v_f - v_m } \\
 &\approxle_{n, \, p} \abs*{f - (v_m, \, \cdot)}_{W^{2, \, p}_\sigma(\esse^n)} \\
 &= \min_{v \in \erre^{n+1}} \abs{f - (v, \, \cdot)}_{W^{2, \, p}_\sigma(\esse^n)}  \\
 &\approxle_{n, \, p} \min_{v \in \erre^{n+1}} \norm{f - (v, \, \cdot)}_{W^{2, \, p}_\sigma(\esse^n)}
\end{align*}
Now we recall that the two Sobolev norms are equivalent, so we obtain 
\[
 \norm*{f - (v_f, \, \cdot)}_{W^{2, \, p}_\sigma(\esse^n)}  \approxle_{n, \, p}  \abs*{f - (v_f, \, \cdot)}_{W^{2, \, p}_\sigma(\esse^n)} 
\]
and the thesis follows.
\end{proof}
\subsection{Proof of the computational lemmas}
We end the article reporting the proof of lemmas \ref{Computations} and \ref{FirstMean}.
\begin{proof}[Proof of lemma \ref{Computations}]
Firstly, we compute the differential of $\psi$:
\begin{equation}\label{Differential}
\restr{d \psi}{x} \daA{T_x \esse^n}{T_{\psi(x)} \Sigma},\ \restr{d \psi}{x}[z] = e^{f (x)} \coup*{z +  \nabla_z f \, x}
\end{equation}
In order to compute the expression for $g$ in $\esse^n$, we fix $x$ in $\esse^n$ and use the usual polar coordinates $\set{\dde{\theta^1} \dots \dde{\theta^n}}$ for the sphere.  
We find
\begin{align*}
g &= g_{ij} d\theta^i \, d\theta^j = \psi^* \restr{\delta}{\Sigma} \coup*{ \dde{\theta^i}, \, \dde{\theta^j}} d\theta^i \, d\theta^j \\
 &= e^{2f} \coup*{ \dde{\theta^i} + \nabla_i f \, x, \, \dde{\theta^j} + \nabla_j f \, x} d\theta^i \, d\theta^j \\
 &= e^{2f} \coup*{ \sigma_{ij} + \nabla_i f \, \nabla_j f } d\theta^i \, d\theta^j
\end{align*}
The expression for $g^{-1}$ follows by direct computation.  \\
Now we compute the normal $\ni = \ni_\Sigma$. Fix $x \in \esse^n$ and consider the system $\set{ \dde{\theta^1} \dots  \dde{\theta^n}, x }$ which is orthogonal in $\erre^{n+1}$. By the definition of $\ni$ we have the relation $\coup*{\ni(x), \, \restr{d\psi}{x}}[z] = 0$ for every $z \in <x>^\perp$. Now we write $\ni = \ni^j \, \dde{\theta^j} + \ni^x \, x$ and obtain 
\[
\abs*{\dde{\theta^j}}^2 \ni_j + \nabla_j f \, \ni^x = 0 \mbox{ for every } j
\]
Normalizing we have
\[
\ni(x) = \frac{1}{\sqrt{1 + \abs{\nabla f}^2}} \coup*{x - \grad_\sigma f(x)} 
\]
which is exactly \eqref{norm}. \\
The expression for $A$ is more complex to compute. Firstly, we easily compute the differential of $\ni$:
\begin{align*}
d\ni \cquad*{\dde{\theta^j}} 
&= \nabla_j \coup*{\frac{1}{\sqrt{1 + \abs{\nabla f}^2}}} \coup{x - \nabla f(x)}  +
\\
&+ \frac{1}{\sqrt{1 + \abs{\nabla f}^2}} \coup*{\dde{\theta^j} - \nabla_j \coup{\nabla f}}
\end{align*}
and now we can make our computation
\begin{align*}
A_{ij} 
:&= \coup*{d \psi \cquad*{\dde{\theta^i}}, \, d\ni\cquad*{\dde{\theta^j}}  } = \\
&= \frac{e^f}{\sqrt{1 + \abs{\nabla f}^2}} \coup*{\dde{\theta^i} + \nabla_i f, \, \dde{\theta^j} - \nabla_j \nabla f}
\\
&=  \frac{e^f}{\sqrt{1 + \abs{\nabla f}^2}} \coup*{\sigma_{ij} - \nabla_i f \, (\nabla_j \nabla f, \, \underbrace{x}_{\ni_{\esse^n}}) - \coup*{\nabla_j \nabla f, \, \dde{\theta^i}} }
\end{align*}
We compute $\nabla_j \nabla f$ in the orthogonal system $\set{\dde{\theta^1} \dots \dde{\theta^n}}$. 
\begin{align*}
(\nabla_j \nabla f, \, x) &= \nabla_j \underbrace{(\nabla f, \, x)}_{= 0} - (\nabla f, \, \nabla_j \ni_{\esse^n} ) = - A_{\esse^n}\coup*{\nabla f, \,  \dde{\theta^j}} = - \nabla_j f \\
\coup*{ \nabla_j \nabla f, \, \dde{\theta^i} } 
&= \nabla_j \underbrace{\coup*{ \nabla f, \, \dde{\theta^i} }}_{ = \partial_i f} - \coup*{\nabla f, \, \nabla_i \dde{\theta^j}} = \partial^2_{ij} f - \Gamma^k_{ij} \partial_k f = \nabla^2_{ij} f  
\end{align*}
We finally write
\[
A_{ij} = \frac{e^f}{\sqrt{1 + \abs {\nabla f}^2}} \coup*{ \sigma_{ij} + \nabla_i f \nabla_j f - \nabla^2_{ij} f}
\]
which is exactly (\ref{Aff}), and we are done. Equality \eqref{Aalta} follows by a direct computation by writing $A^i_j = g^{li} A_{lj}$ and we do not report it. \\ 
Formula \eqref{Volume} follows from the area formula (see \cite{AFP}): 
\[
\int_\Sigma h(y) \, dV_g(y) = \int_{\esse^n} h(\psi(x)) J_{d\psi}(x) \, dV_\sigma \mbox{ for any } h \in C(\Sigma)
\] 
with 
\[
J_{d\psi}(x)^2 = \det \restr{d^*\psi}{x} \circ \restr{d\psi}{x}
\]
where $d^* \psi$ is the adjoint differential, whose representative matrix is simply the transpose of the $d \psi$ representative matrix. Taking $\set{\dde{\theta^1} \dots \dde{\theta^n}, \, x}$ as frame for $\erre^{n+1}$ we easily find the expression
\[
\det \restr{d^*\psi}{x} \circ \restr{d\psi}{x} = e^{2 n f} \coup*{1 + \abs{\nabla f}^2}
\]
and the result follows simply by taking the square root. 

Lastly we deal with the Christoffel symbols. We recall the formula 
\[
\Gammag_{ij}^k = \frac{1}{2}g^{ks} \coup*{ \partial_i g_{js} + \partial_j g_{is} - \partial_s g_{ij}}
\]
and now we expand it:
\begin{align*}
\Gammag_{ij}^k 
&= \frac{1}{2} \coup*{ \sigma^{ks} - \frac{\nabla^k f \nabla^s f}{1 + \abs{\nabla f}^2} } \coup*{ \partial_i \sigma_{js} + \partial_j \sigma_{is} - \partial_s \sigma_{ij} } + \\
& + \frac{1}{2} \coup*{ \sigma^{ks} - \frac{\nabla^k f \nabla^s f}{1 + \abs{\nabla f}^2} } \coup*{ \partial_i \coup{\partial_j f \partial_s f} + \partial_j \coup{\partial_i f \partial_s f} - \partial_s \coup{\partial_i f \partial_j f} } + \\
& + \frac{1}{2} g^{ks} \coup*{ \nabla_i f \, g_{js} + \nabla_j f \, g_{is} - \nabla_s f \, g_{ij} } \\
&= \frac{1}{2} \coup*{ \sigma^{ks} - \frac{\nabla^k f \nabla^s f}{1 + \abs{\nabla f}^2} } \coup*{ \partial_i \sigma_{js} + \partial_j \sigma_{is} - \partial_s \sigma_{ij} + 2 \, \partial^2_{ij} f \, \partial_s f }+ \\
& + \frac{1}{2} \coup*{ \nabla_i f \, \delta_j^k + \nabla_j f \, \delta_i^k - \nablag_k f \, g_{ij} } \\
&= \Gamma_{ij}^k + \partial^2_{ij} f \, \partial^k f - \frac{\abs{\nabla f}^2}{1 + \abs{\nabla f}^2} \partial^2_{ij} f \, \partial^k f + \\
&- \frac{1}{1 + \abs{\nabla f}^2}  \partial^k f  \, \partial_s f \, \frac{\sigma^{ls}}{2} \coup*{ \partial_i \sigma_{js} + \partial_j \sigma_{is} - \partial_s \sigma_{ij} }+ \\ 
& + \frac{1}{2} \coup*{ \nabla_i f \, \delta_j^k + \nabla_j f \, \delta_i^k - \nablag_k f \, g_{ij} } \\
&= \Gamma_{ij}^k + \frac{1}{1 + \abs{\nabla f}^2}  \coup*{\partial^2_{ij} f - \partial_s f \, \Gamma_{ij}^s} \partial^k f + \\
& + \frac{1}{2} \coup*{ \nabla_i f \, \delta_j^k + \nabla_j f \, \delta_i^k - \nablag_k f \, g_{ij} } \\
&= \Gamma_{ij}^k + \frac{1}{1 + \abs{\nabla f}^2} \nabla^2_{ij} f \nabla^k f + \frac{1}{2} \coup*{ \nabla_i f \, \delta_j^k + \nabla_j f \, \delta_i^k - \nablag_k f \, g_{ij} }
\end{align*}
\end{proof}

Lastly, we prove corollary \ref{FirstMean}.
\begin{proof}[Proof of corollary \ref{FirstMean}]
Let $\lambda_0 \in \erre$ such that 
\[
\norm{A - \lambda_0 \, g}_p = \min_{\lambda \in \erre} \norm{A - \lambda \, g}
\]
We simply write
\begin{align*}
\norm{A - \overline{H} \, g}_p 
&= \norm{ (A - \lambda_0 \, g) + (\lambda_0 - \overline{H}) \, g}_p \\
&\le \norm{A - \lambda_0 \, g}_p + n \, \Vol_n(\esse^n)^{ \frac{1}{p}} \abs{\overline{H} - \lambda_0}
\end{align*}
The thesis follows estimating the last term.
\begin{align*}
\abs{\overline{H} - \lambda_0} 
&= \abs{ \fint_{\esse^n} \frac{1}{n} \tr_g A \, dV_g - \lambda_0 } =  \frac{1}{n} \abs{ \fint g^{ij} \, \coup*{ A_{ij} - \lambda_0 \, g_{ij} } \, dV_g } \\
&= \frac{1}{n \Vol_n(\esse^n)}  \abs{ \int \coup*{g, \, A - \lambda_0 \, g} \, dV_g } 
\le  \frac{1}{n \Vol_n(\esse^n)} \norm{g}_{p'} \norm{A - \lambda_0 \, g}_p  \\
& = \Vol_n(\esse^n)^{-\frac{1}{p}}\norm{A - \lambda_0 \, g}_p
\end{align*}
We have obtained
\[
\norm{A - \overline{H} \, g}_p \le (n+1) \, \min_{\lambda \in \erre} \norm{A - \lambda \, g}_p
\]
which is exactly the thesis.
\end{proof}
\section*{Acknowledgements}
The author is thankful to Prof. Camillo De Lellis who brought this problem to his attentions and gave him the main idea for the proof. He also wish to thanks Alessandro Pigati and Elia Brué for their valuable suggestions and comments on this article.

\bibliographystyle{plain}
\bibliography{Problem}
\end{document}